\documentclass{amsart}
\usepackage{amsmath}
\usepackage{amssymb}
\usepackage{amsfonts}
\usepackage{latexsym}
\usepackage{bm}
\usepackage{color}
\usepackage{graphicx}

\newcommand{\be}{\begin{equation}}
\newcommand{\en}{\end{equation}}
\newcommand{\bea}{\begin{eqnarray}}
\newcommand{\ena}{\end{eqnarray}}
\newcommand{\beano}{\begin{eqnarray*}}
\newcommand{\enano}{\end{eqnarray*}}
\newcommand{\bee}{\begin{enumerate}}
\newcommand{\ene}{\end{enumerate}}

\newcommand{\ad}{^{\mbox{\scriptsize $\dag$}}}

\newcommand{\mc}{\mathcal}
\newcommand{\mb}{\mathbb}

\newcommand{\Id}{1\!\!1}

\newcommand{\cF}{{\mathcal F}}
\newcommand{\B}{{\mc B}}

\newcommand{\Lc}{{\mc L}}

\newcommand{\D}{{\mc D}}

\newcommand{\E}{{\mc E}}

\newcommand{\cM}{{\mathcal M}}

\newcommand{\up}{\upharpoonright}

\newtheorem{defn}{Definition}[section]
\newtheorem{thm}[defn]{Theorem}
\newtheorem{prop}[defn]{Proposition}

\newtheorem{example}[defn]{Example}

\newtheorem{rem}[defn]{Remark}

\def\x{\relax\ifmmode {\mbox{*}}\else*\fi}
\newcommand{\beex}{\begin{example}$\!\!${\bf }$\;$\rm }
\newcommand{\enex}{ \end{example}}
\newcommand{\berem}{\begin{rem}$\!\!${\bf }$\;$\rm }
\newcommand{\enrem}{ \end{rem}}
\newcommand{\bedefi}{\begin{defn}$\!\!${\bf }$\;$\rm }
\newcommand{\findefi}{\end{defn}}

\newcommand{\ip}[2]{\left\langle {#1}\left|{#2}\right.\right\rangle}

\newcommand{\gl}{{\mathfrak L}}
\newcommand{\LDD}{\gl(\D,\D^\times)}

\newcommand{\LD}{\Lc^\dagger(\D)}

\def\H{{\mathcal H}}
\def\K{{\mathcal K}}

\newcommand{\restr}[1]{_{\!\up{#1}}}
\def\LG{{\mathfrak L}}

\newcommand{\sH}{{\sf{H}}}

\def\S{\mathcal S}
\begin{document}
\title[Riesz-like bases]
{Riesz-like bases in rigged Hilbert spaces}
\author{Giorgia Bellomonte}
\author{Camillo Trapani}
\address{Dipartimento di Matematica e Informatica,
Universit\`a degli Studi di Palermo, I-90123 Palermo, Italy}
\email{giorgia.bellomonte@unipa.it}
\email{camillo.trapani@unipa.it}


\begin{abstract}{ The notions of Bessel sequence, Riesz-Fischer sequence
and Riesz basis are generalized to a rigged Hilbert space $\D[t]
\subset \H \subset \D^\times[t^\times]$. A Riesz-like basis, in
particular, is obtained by considering a sequence $\{\xi_n\}\subset
\D$ which is mapped by a one-to-one continuous operator
$T:\D[t]\to\H[\|\cdot\|]$ into an orthonormal basis of the central
Hilbert space $\H$ of the triplet. The operator $T$ is, in general,
an unbounded operator in $\H$. If $T$ has a bounded inverse then the
rigged Hilbert space is shown to be equivalent to a triplet of
Hilbert spaces.} 
\end{abstract} \maketitle

\section{Introduction}\label{sect_Introduction}

{Riesz bases {(i.e., sequences of elements $\{\xi_n\}$ of a Hilbert
space $\H$ which are transformed into orthonormal bases by some
bounded operator with bounded inverse)} often appear as eigenvectors
of nonself-adjoint operators. The simplest situation is the
following one. Let $H$ be a self-adjoint operator with discrete
spectrum defined on a subset $D(H)$ of the Hilbert space  $\H$.
Assume, to be more definite, that each eigenvalue $\lambda_n$ is
simple. Then the corresponding eigenvectors $\{e_n\}$ constitute an
orthonormal basis of $\H$. If $X$ is another operator {\em similar}
to $H$, i.e., there exists a bounded operator $T$ with bounded
inverse $T^{-1}$ which intertwines $X$ and $H$, in the sense that
$T: D(H)\to D(X)$ and $XT\xi=TH\xi$, for every $\xi \in D(H)$, then,
as it is easily seen, the vectors $\{\varphi_n\}$ with
$\varphi_n=Te_n$ are eigenvectors of $X$ and constitute a Riesz
basis for $\H$. There are, however, more general situations, mostly
coming from physical applications, where the intertwining operator
$T$ exists but at least one between $T$ and $T^{-1}$ is unbounded.
This is actually the case of the so-called cubic Hamiltonian $X=
p^2+ix^3$ of Pseudo-Hermitian Quantum Mechanics, for which it has
been proved that there is no intertwining operator bounded with
bounded inverse which makes it similar to a self-adjoint operator
\cite{siegl}. Of course, for studying these cases, one also has to
relax the notion of similarity since problems of domain may easily
arise (see \cite{JPA_CT1, JPA_CT2} for a full discussion of the
various notions of (quasi-) similarity that one may introduce).

Also, when studying the formal commutation relation $[A,B]=\Id$,
where $B$ is not the adjoint of $A$ (the so-called {\em
pseudo-bosons} studied by Bagarello \cite{bagarello, bagarello2}),
in the most favorable situation,  one finds two biorthogonal
families of vectors $\{\phi_k\}, \{\psi_k\}$, a positive
intertwining operator $K$  ($K\varphi_n=\psi_n$, $n \in {\mb N}$)
and the family $\{e_n\}$ with $e_n=K^{-1/2}\varphi_n$, $n \in {\mb
N}$, is an orthonormal family of vectors. But, in general both $K$
and $K^{-1}$ are unbounded \cite{bit1,bit2}.

These examples motivate, in our opinion, a study of possible
generalizations of the notion of Riesz basis that could cover these
situations of interest for applications.

{Whenever unbounded operators are involved, dealing with
discontinuity and with sometimes nontrivial domain problems becomes
unavoidable. Both difficulties can be by-passed if one enlarges the
set-up from Hilbert spaces to {\em rigged Hilbert spaces.}

 A {rigged Hilbert space} (RHS) consists of a triplet $(\D
,\H, \D^\times)$ where $\D$ is a dense subspace of $\H$ endowed with
a topology $t$, finer than that induced by the Hilbert norm of $\H$,
and $\D^\times$ is the conjugate dual of $\D[{t}]$, endowed with the
strong topology ${t}^\times:= \beta(\D^\times, \D)$.

Of course, one could also pose the problem of extending the notion
of Riesz basis in the more general set-up of locally convex spaces,
but the nature itself of the notion of Riesz basis requires {also} a
control of its behavior in the context of duality and, as we shall
see, a {\em Riesz-like} basis on a locally convex space $\D[t]$ will
automatically make of $\D$ the {\em smallest} space of rigged
Hilbert space. Thus it appears natural to consider rigged Hilbert
spaces from the very beginning.

On the other hand,} rigged Hilbert spaces (and their further
generalizations like e.g. partial inner product spaces)  have plenty
of applications. In Analysis they provide the general framework for
distribution theory; in Quantum Physics they give a convenient
description of the Dirac formalism \cite[Ch. 7]{jpact_book}.
Finally, rigged Hilbert spaces (e.g. those generated by the
Feichtinger algebra) or lattices  of Hilbert or Banach spaces
(mixed-norm spaces, amalgam spaces, modulation spaces) play also an
important role in signal analysis (see \cite[Ch. 8]{jpact_book}, for
an overview).

As it is known, a Riesz basis {$\{\xi_n\}$} in a Hilbert space $\H$
is also a {\em frame} \cite{casazza, Christensen, heil}; i.e., there
exist positive numbers $c, C$ such that
\begin{equation}\label{eqn_bounds} c\|\xi\|^2 \leq \sum_{n=1}^\infty |\ip{\xi}{\xi_n}|^2 \leq C \|\xi\|^2, \quad \forall \xi \in \H.\end{equation}
The peculiarity of a Riesz basis relies in its exactness or
minimality: a frame is a Riesz basis if it ceases to be a frame when
anyone of its elements is dropped out. } The notion of frame is
crucial in signal analysis and for coherent states (see, e.g.
\cite{casazza} and references therein) and in approximation theory
\cite{vinti3, vinti2, vinti1}. { A further generalization is the
notion of {\em semi-frame} \cite{JPA_balazs} for which one of the
above frame bounds is absent (lower or upper semi-frames). For
instance a  lower semi-frame has an {\em unbounded} frame operator,
with {\em bounded} inverse.}

The paper is organized as follows.

    In Section
 \ref{sect_one}, after some preliminaries,  we discuss shortly the notion of basis in a rigged Hilbert space $\D[t] \subset \H \subset \D^\times[t^\times]$. Then we introduce {\em Bessel-like} sequences in $\D^\times[t^\times]$ and {\em  Riesz-Fischer-like} sequences in $\D[t]$ and study, in the present context, their interplay in terms of duality.
  In Section \ref{sect_weak Riesz bases}, we define $\{\xi_n\}$ to be a {\em Riesz-like basis} if there exists a one-to-one linear map $T: \D\to\H$, continuous from $\D[t]$ into $\H[\|\cdot\|]$, such that $\{T\xi_n\}$ is an orthonormal basis for the central Hilbert space $\H$. Some characterizations of these bases are given.
 Finally, we consider the special case where $T$ has also a continuous inverse. 
     This additional assumption, even though natural, reveals to be quite strong, since, as we will see, the rigged Hilbert space $\D[t] \subset \H \subset \D^\times[t^\times]$ is in fact equivalent to a triplet of Hilbert spaces. An application to nonself-adjoint Hamiltonians is briefly discussed in Section \ref{sect_appl}.

\section{Preliminaries and basic aspects} \label{sect_one}
\subsection{Rigged Hilbert spaces and operators on them}

   Let $\D$ be a dense subspace of  $\H$.  A locally convex topology $t$ on $\D$ finer than the topology induced by the Hilbert norm defines, in standard fashion,
a {\em rigged Hilbert space} (RHS)
\begin{equation}\label{eq_one_intr}
\D[t] \hookrightarrow  \H \hookrightarrow\D^\times[t^\times],
\end{equation}
where $\D^\times$  is the vector space of all continuous conjugate
linear functionals on $\D[t]$, i.e., the conjugate dual of $\D[t]$,
endowed with the {\em strong dual topology} $t^\times=
\beta(\D^\times,\D)$ and $\hookrightarrow $ denotes a continuous
embedding. Since the Hilbert space $\H$ can be identified  with a
subspace of $\D^\times[t^\times]$, we will systematically read
\eqref{eq_one_intr} as a chain of topological inclusions: $\D[t]
\subset  \H \subset\D^\times[t^\times]$.  These identifications
imply that the sesquilinear form $B( \cdot , \cdot )$ that puts $\D$
and $\D^\times$ in duality is an extension of the inner product of
$\D$;
 i.e. $B(\xi, \eta) = \ip{\xi}{\eta}$, for every $\xi, \eta \in \D$ (to simplify notations we adopt the symbol $\ip{\cdot}{\cdot}$ for both of
 them).

{\beex \label{ex_triplet}Let $T$ be a closed densely defined
operator with domain $D(T)$ in Hilbert space $\H$. Let us endow
$D(T)$ with the graph norm $\|\cdot \|_T$ defined by
 $$ \|\xi\|_T= (\|\xi\|^2 + \|T\xi\|^2)^{1/2}= \|(I+T^* T)^{1/2}\xi\|, \quad \xi \in D(T).$$ With this norm $D(T)$ becomes a Hilbert space, denoted by $\H_T$. If $\H_T^\times$ denotes the Hilbert space conjugate dual of $\H_T$, then we get the triplet of Hilbert spaces $$ \H_T \subset \H \subset \H_T^\times$$ which is a particular example of rigged Hilbert space.
 \enex}

\beex \label{ex_rhs} Let $\D$ be a dense domain in Hilbert space
$\H$ and denote by $\LD$ the *-algebra consisting of all closable
operators $A$ with $D(A)=\D$, which together with their adjoints,
$A^*$, leave $\D$ invariant. The involution of $\LD$ is defined by
$A\mapsto A^\dag$, where $A^\dag= A^*\, \restr{\D}$. The *-algebra
$\LD$ defines in $\D$ the {\em graph topology} $t_\dag$ by the
family of seminorms
$$ \xi\in \D \to \|\xi\|_A:= \|(I+A^* \overline{A})^{1/2}\xi\|, \quad A \in \LD.$$
Since the topology $t_\dag$ is finer than the topology induced on
$\D$ by the Hilbert norm of $\H$, it defines in natural way a
structure of rigged Hilbert space. \enex

Let now $\D[t] \subset \H \subset \D^\times[t^\times]$ be a rigged
Hilbert space,
 and let $\LDD$ denote the vector space of all continuous linear maps from $\D[t]$ into  $\D^\times[t^\times]$. {If $\D[t]$ is barreled (e.g., reflexive)}, an involution $X \mapsto X\ad$ can be introduced in $\LDD$  by the equality
$$ \ip{X\ad \eta}{ \xi} = \overline{\ip{X\xi}{\eta}}, \quad \forall \xi, \eta \in \D.$$  Hence, in this case, $\LDD$ is a $^\dagger$-invariant vector space.

If $\D[t]$ is a {smooth} space (e.g., Fr\'echet and reflexive), then $\LDD{}$ is a quasi *-algebra over $\LD$ \cite[Definition 2.1.9]{ait_book}.\\

Let $\E, \cF \in \{\D, \H, \D^\times\}$ and $\LG(\E,\cF)$ the space
of all continuous linear maps from $\E[t_\E]$ into $\cF[t_\cF]$. We
put {$$\mc{C}(\E, \cF):=\{ X\in \LDD:\, \exists Y\in  \gl(\E, \cF),
\, Y\xi=X\xi, \forall \xi \in \D\}.$$} In particular, if $X \in
C(\D,\H)$ then its adjoint $X^\dag \in \LDD$ has an extension from
$\H$ into $\D^\times$, which we denote by the same symbol.

\smallskip The space $\LDD$ has been studied at length by several authors (see, e.g. \cite{kurst1,kurst2,kurst3, tratschi}) and several pathologies concerning their multiplicative structure have been considered (see also \cite{ait_book, jpact_book} and references therein). Recently some spectral properties of operators of these classes have also been studied \cite{BDT_JMAA}.

\subsection{Topological bases and Schauder bases}
{Let $\E[t_\E]$ be a locally convex space and $\{\xi_n\}$  a
sequence of vectors of $\E$}. We adopt the following terminology:
\begin{itemize}
\item[(i)]the sequence $\{\xi_n\}$ is {\em complete} or {\em total} if the linear span of $\{\xi_n\}$ is dense in $\E[t_\E]$;
\item[(ii)]the sequence $\{\xi_n\}$ is a {\em topological basis} for $\E$ if, for every $\phi \in \E$, there exists  a {\em unique} sequence $\{c_n\}$ of complex numbers such that
\begin{equation}\label{eq_develop} \phi = \sum_{n=1}^\infty c_n \xi_n, \end{equation}
where the series on the right hand side converges in $\E[t_\E]$.
\end{itemize} Every coefficient $c_n=c_n(\phi)$ in
\eqref{eq_develop} can be regarded as a linear functional on $\E$
and, following \cite{jarchow}, we say that
\begin{itemize}\item[(iii)] a topological basis $\{\xi_n\}$ of $\E[t_\E]$ is a {\em Schauder basis}  if the coefficient functionals $\{c_n\}$ are $t_\E$-continuous.
\end{itemize}

\berem \label{rem_23} We notice the following well-known facts.
\begin{itemize}

\item[(a)] If $\E$ has a total sequence, then it is a separable space.

\item[(b)]Every topological basis is a complete sequence; the converse is  false, in
general.

\item[(c)]  If $\{\xi_n\} $ is a topological basis for $\E$, then $\{\xi_n\}$ is {\em $\omega$-independent}; i.e., if $\sum_{n=1}^\infty c_n \xi_n=0$, then $c_n=0$, for every $n \in {\mb N}$. This in turn implies that the sequence $\{\xi_n\} $ consists of linearly independent vectors.

\item[(d)]If $\E[t_\E]$ is a Fr\'echet space, then every topological basis is a Schauder basis (\cite[Sect. 14.2, Th. 5]{jarchow}). \end{itemize}\enrem

By a slight modification of \cite[Sect. 14.3, Th. 6]{jarchow} we
have
\begin{prop} \label{prop_3.2}A complete sequence of vectors $\{\xi_n\}\subset \E$ is a Schauder basis of $\E[t_\E]$ if, and only if, for every $n \in {\mb N}$ and every
continuous seminorm $p$ on $\E[t_\E]$, there exists a continuous
seminorm $q$ on $\E[t_\E]$ such that
\begin{equation} \label{eq_one} p\left( \sum_{i=1}^n c_i \xi_i \right) \leq q\left( \sum_{i=1}^{n+m} c_i \xi_i \right)\end{equation}
where $c_1, \ldots, c_{n+m}$ are arbitrary complex numbers and $ m$
is an arbitrary natural number.
\end{prop}
{ As is known, a Riesz basis $\{\xi_n\}$ in Hilbert space $\H$ is
transformed by some bounded operator into an orthonormal basis of
$\H$; this is equivalent to saying that a new (and equivalent) inner
product can be introduced in $\H$ which makes of $\{\xi_n\}$ an
orthonormal basis. A similar notion for locally convex spaces, calls
immediately on the stage rigged Hilbert spaces.

\begin{prop}\label{prop_newrhs}
Let $\{\xi_n\}\subset \E$ be a Schauder basis of $\E[t_\E]$ and
assume that there exists a one-to-one continuous linear map $T$ from
$\E[t_\E]$ into some Hilbert space $\K[\|\cdot\|]$ such that
$\{T\xi_n\}$ is an orthonormal basis of $\K$. Then there exists an
inner product $\ip{\cdot}{\cdot}_+$ on $\E \times \E$ such that the
topology induced on $\E$ by the norm $\|\cdot\|_+$ is coarser than
$t_\E$ and $\{\xi_n\}$ is an orthonormal basis.
\end{prop}
\begin{proof} Define $\ip{\xi}{\eta}_+:=\ip{T\xi}{T\eta}$, $\xi, \eta \in \E$. Then all the statements follow immediately.

\end{proof}
Then, under the conditions of Proposition \ref{prop_newrhs},  one
can consider $\E$ as a subspace of the Hilbert space completion
$\H_+$ of $\E[\|\cdot\|_+]$, so that a rigged Hilbert space can be
built in natural way: $\E[t_\E] \subset \H_+\subset
\E^\times[t_\E^\times]$. This is essentially the reason why, as
announced in the Introduction, we will confine ourselves within this
framework.}

\bigskip
Let $\D[t] \subset \H \subset \D^\times[t^\times]$ be a rigged
Hilbert space and $\{\xi_n\}$ a Schauder basis for $\D[t]$. Then,
every $f \in \D$ can be written as $\sum_{n=1}^\infty c_n(f) \xi_n$,
for uniquely determined suitable coefficients $c_n(f)$. Since every
$c_n$ is a {continuous} linear functional on $\D[t]$, there exists a
sequence $\{\zeta_n\} \subset \D^\times$ such that
$$ c_n(f) = \overline{\ip{\zeta_n}{f}}, \quad \forall n\in {\mb N}, f \in \D.$$
For every $n \in {\mb N}$, the vector $\zeta_n$ is uniquely
determined. If we take $f= \xi_k$, then it is clear that
$c_n(\xi_k)=\delta_{n,k}$. Hence $\ip{\zeta_n}{\xi_k}=
\delta_{n,k}$; i.e., the sequences $\{\xi_n\}$ and $\{\zeta_n\}$ are
{\em biorthogonal}.

More precisely,
{ \begin{prop} Let $\{\xi_n\}$ be a topological basis for $\D[t]$. The following statements are equivalent.
\begin{itemize}
\item[(i)] $\{\xi_n\}$ is a Schauder basis.
\item[(ii)] $\{\xi_n\}$ is minimal; i.e., $\xi_k \not\in \overline{{\rm span}\{\xi_m; m\neq k\}}^t$, for every $k \in {\mb N}$.
\item[(iii)] There exists a unique sequence $\{\zeta_n\}\subset \D^\times$ such that $\{\xi_n\}$ and $\{\zeta_n\}$ are biorthogonal.
\end{itemize}

\end{prop}
$({\rm i}) \Leftrightarrow ({\rm ii})$ is proved in \cite[Sect.14.2,
Prop. 3]{jarchow} and $ ({\rm ii}) \Rightarrow ({\rm iii})$ in
\cite[Sect.14.2, Prop. 1]{jarchow};  $({\rm iii}) \Rightarrow ({\rm ii})$ is trivial. See also \cite[Th.
6.1.1]{Christensen}.}

\begin{prop}\label{prop_2.4.r} Let $\{\xi_n\}$ be a Schauder basis for $\D[t]$. Then there exists a sequence $\{\zeta_n\}$ of vectors of $\D^\times$ such that
\begin{itemize}
\item[(i)] the sequences $\{\xi_n\}$ and $\{\zeta_n\}$ are { biorthogonal};
\item[(ii)] for every $f\in \D$,
\begin{equation}\label{eq three} f= \sum_{n=1}^\infty \overline{\ip{\zeta_n}{f}}\xi_n;\end{equation}
\item[(iii)]The partial sum operator $S_n$, given by
$$S_n f = \sum_{k=1}^n \overline{\ip{\zeta_k}{f}}\xi_k, \quad f\in \D ,$$
is continuous from $\D[t]$ into $\D[t]$ and has an adjoint
$S_n^\dag$ everywhere defined in $\D^\times$ given by
$$ S^\dag_n \Psi= \sum_{k=1}^n {\ip{\Psi}{\xi_k}}\zeta_k, \quad \Psi \in \D^\times.$$
\end{itemize}
\end{prop}
{The proof is straightforward. }

\smallskip

\begin{prop} \label{prop_1.10} Let $\{\xi_n\}$ be a Schauder basis for $\D[t]$. Then, the following statements hold.
\begin{itemize}

\item[(i)] The sequence $\{\zeta_n\}$ in  \eqref{eq three} is complete in $\D^\times[\tau]$, where $\tau$ is a topology of the conjugate dual pair $(\D^\times, \D)$. If $\D[t]$ is reflexive, $\{\zeta_n\}$ is complete also with respect to $t^\times$.

    \item[(ii)] The sequence $\{\zeta_n\}$  is a basis for $\D^\times$ with respect to the weak topology; i.e., if $\Psi \in \D^\times$ one has
    \begin{equation}\label{eq_two}\ip{\Psi}{f}= \ip{\sum_{k=1}^\infty {\ip{\Psi}{\xi_k}}\zeta_k}{f}= \sum_{k=1}^\infty {\ip{\Psi}{\xi_k}}\ip{\zeta_k}{f}, \quad \forall f \in \D.\end{equation}
    \end{itemize}
\end{prop}

\begin{proof}
(i): Assume that $\{\zeta_n\}$ is not complete. Then there exists
$f\neq 0$, $ f \in \D$ (regarded as the conjugate dual of
$\D^\times[\tau]$) such that $\ip{\zeta_n}{f}=0$, for every $n \in
{\mb N}$. From \eqref{eq three} it follows that $f=0$, a
contradiction. If $\D[t]$ is reflexive, the statement follows from
the equality of $t$ and the Mackey topology $\tau(\D^\times, \D)$.

(ii): Assume first that $\Phi \in \D^\times$ is of the form $\Phi=
\sum_{k=1}^n c_k \zeta_k$. Then it is easily seen that $S_n^\dag
\Phi = \Phi$. Now, if $\Psi\in \D^\times$, for every $f \in \D$ and
for every $\epsilon>0$, there exists $\Phi=\sum_{k=1}^n c_k \zeta_k$
such that $|\ip{\Psi- \Phi}{f}|<\epsilon$. On the other hand, since
$S_n f \to f$, there exists $n_\epsilon \in {\mb N}$ such that for
$n>n_\epsilon$, $|\ip{\Psi -\Phi}{S_nf-f}|<\epsilon$. Thus we have

\begin{align*} \left|\ip{S_n^\dag \Psi - \Psi}{f}\right| &\leq \left|\ip{S_n^\dag \Psi - S_n^\dag\Phi}{f}\right| + \left|\ip{S_n^\dag \Phi - \Phi}{f}\right|+|\ip{\Phi - \Psi}{f}|\\
& =|\ip{ \Psi - \Phi}{S_n f}|+|\ip{\Phi - \Psi}{f}|\\
&\leq |\ip{ \Psi - \Phi}{S_n f-f}|+2|\ip{\Phi - \Psi}{f}|<3\epsilon.
\end{align*}
Hence $S_n^\dag \Psi\to \Psi$ weakly, or
$$ \ip{\Psi}{f}= \ip{\sum_{k=1}^\infty {\ip{\Psi}{\xi_k}}\zeta_k}{f}= \sum_{k=1}^\infty {\ip{\Psi}{\xi_k}}\ip{\zeta_k}{f}.$$

\end{proof}

For  $f\in \D \subset \D^\times$,  \eqref{eq_two} gives in
particular
$$\|f\|^2 = \sum_{k=1}^\infty {\ip{f}{\xi_k}}\ip{\zeta_k}{f}, \quad \forall f \in \D;$$
so that the series on the right hand side is convergent, for every
$f \in \D$.

\berem There is a wide interest and a rich literature on bases or
frames in locally convex spaces (in particular, Banach spaces)
 and on their existence, see, e.g. \cite{casazzachris, pilipovic} and \cite{bonet} and references therein. \enrem

\subsection{Bessel- and Riesz-Fischer-like sequences}
We assume, from now on, that $\D[t]$ is complete and reflexive.
\bedefi Let $\{\zeta_n\}$ be a sequence in $\D^\times$. We say that
$\{\zeta_n\}$ is a {\em Bessel-like sequence} if, for every bounded
subset $\cM$ of $\D[t]$,
\begin{equation}\label{eqn_bessel} \sup_{\eta\in \cM}\sum_{k=1}^\infty|\ip{\zeta_k}{\eta}|^2=:\gamma_{\cM}<\infty.\end{equation}
\findefi

\begin{prop}\label{prop_210} A sequence $\{\zeta_n\}$ of elements of $\D^\times$ is Bessel-like if and only if
$$\sum_{k=1}^\infty|\ip{\zeta_k}{\eta}|^2< \infty, \quad \forall \eta\in \D$$
and the {\em analysis operator}
$$F: \eta \in \D[t] \to \{\overline{\ip{\zeta_k}{\eta}}\}\in \ell^2[\| \cdot \|_2]$$
is continuous.
\end{prop}
\begin{proof} Let $\{\zeta_n\}$ be Bessel-like. From \eqref{eqn_bessel} it is clear that for every $\eta \in \D$, $\sum_{k=1}^\infty|\ip{\zeta_k}{\eta}|^2< \infty$.

Now we prove that $$U: \{a_n\}\in \ell^2 \to \sum_{n=1}^\infty
a_n\zeta_n $$ is a well-defined continuous linear map from
$\ell^2[\|\cdot\|_2]$ into $\D^\times[t^\times]$.

 We begin with proving that $\sum_{n=1}^\infty a_n\zeta_n$ converges in $\D^\times[t^\times]$. Let $\cM$ be a bounded subset of $\D[t]$. Then, for $n>m$,
\begin{align*}
\sup_{\eta\in \cM}\left|\ip{\sum_{k=1}^n a_k\zeta_k - \sum_{k=1}^m a_k\zeta_k}{\eta} \right| &= \sup_{\eta\in \cM}\left|\ip{\sum_{k=m+1}^n a_k\zeta_k}{\eta}\right|\\ &\leq \sup_{\eta\in \cM}\sum_{k=m+1}^n |a_k\ip{\zeta_k}{\eta}|\\
&\leq \left(\sum_{k=m+1}^n |a_k|^2\right)^{1/2} \cdot \sup_{\eta\in
\cM}\left( \sum_{k=1}^\infty|\ip{\zeta_k}{\eta}|^2\right)^{1/2}\\ &
\leq \gamma_{\cM}^{1/2} \left(\sum_{k=m+1}^n |a_k|^2\right)^{1/2}
\to 0, \; \mbox{as }n,m\to \infty.
\end{align*}
Hence the partial sums $\sum_{k=1}^n a_k\zeta_k$ constitute a Cauchy
sequence in $\D^\times[t^\times]$ and, since $\D^\times[t^\times]$
being {reflexive is\footnote{A locally convex space space is said to
be quasi-complete if every closed bounded subset is complete.}
quasi-complete \cite[Ch. IV, 5.5, Cor.1]{Schaefer} }, the series
converges in $\D^\times$. Moreover, by simple modifications of the
previous inequalities it follows also that
$$ \sup_{\eta\in \cM}\left|\sum_{k=1}^\infty a_k\ip{\zeta_k}{\eta}\right|\leq \gamma_{\cM}^{1/2} \|\{a_n\}\|_2 .$$
Thus, $U$ is continuous from $\ell^2[\|\cdot\|_2]$ into
$\D^\times[t^\times]$ and therefore by the reflexivity of $\D[t]$,
$U$ has a continuous adjoint map $U^\dag: \D[t] \to
\ell^2[\|\cdot\|_2]$. It is easily checked that $$U^\dag \eta=
\{\overline{\ip{\zeta_n}{\eta}}\}, \quad \forall \eta \in \D.$$ Thus
$U^\dag =F$ and $F$ is continuous.

Conversely, let us assume that $\{\ip{\zeta_k}{\eta}\}\in \ell^2$
and that $F$ is continuous. This implies that there exists a
continuous seminorm $p$ on $\D[t]$ such that
$$ \|F \eta\|_2 = \left(\sum_{k=1}^\infty|\ip{\zeta_k}{\eta}|^2 \right)^{1/2} \leq p(\eta), \quad \forall \eta \in \D.$$
Thus, if $\cM$ is a bounded subset of $\D[t]$, we get
$$ \sup_{\eta\in \cM}\sum_{k=1}^\infty|\ip{\zeta_k}{\eta}|^2 = \sup_{\eta\in \cM}\|F \eta\|_2^2 \leq \sup_{\eta\in \cM}p(\eta)^2<\infty, \quad \forall \eta \in \D.$$
Hence $\{\zeta_n\}$ is a Bessel-like sequence.
\end{proof}

As usual, we will call the operator $F^\dag
{:\{a_n\}\in\ell^2\to\sum_{n=1}^\infty a_n\zeta_n\in\D^\times}$, the
{\em syntesis} operator of the sequence $\{\zeta_n\}$.

From Proposition \ref{prop_210} and from the fact that
\eqref{eqn_bessel} is not affected from a possible reordering of the
elements $\{\zeta_n\}$ it follows that if $\{\zeta_n\}$ is a
Bessel-like sequence and $\{a_n\}\in \ell^2$ then the series
$\sum_{n=1}^\infty a_n\zeta_n$ converges {\em unconditionally} in
$\D^\times[t^\times]$.

If $\{\zeta_n\}$ is a Bessel-like sequence, then the operator
$F^\dag F$ (we keep for it the name of {\em frame operator}, as
usual) is a continuous linear map from $\D[t]$ into
$\D^\times[t^\times]$; i.e. $F^\dag F\in \LDD$. Clearly,
$$F^\dag F \eta= \sum_{k=1}^\infty \ip{\zeta_k}{\eta}\zeta_k$$ where the series on the right hand side converges in $\D^\times[t^\times]$. The operator $F^\dag F$ is positive, in the sense that $\ip{F^\dag F \eta}{\eta}\geq 0$, for every $\eta \in \D$.

\berem A sequence $\{\xi_n\}$ of elements of $\D$ can also be
considered as a sequence in $\D^\times$. Hence the notion of
Bessel-like sequence can be given also in this case, and analysis
and synthesis operators act in the very same way as before.
Moreover, if both series $\sum_{k=1}^\infty a_k \xi_k$ and
$\sum_{k=1}^\infty a_{\sigma(k)} \xi_{\sigma(k)}$ converge in
$\D[t]$, where $\sigma: {\mb N}\to {\mb N}$  is a bijection, then
they have the same sum, since $\sum_{k=1}^\infty a_k \xi_k$
converges unconditionally in $\D^\times[t^\times]$.

\enrem {
\begin{prop} \label{prop_211}A sequence $\{\zeta_n\}$ of elements of $\D^\times$ is Bessel-like if and only if, for every orthonormal basis $\{e_n\}$ in $\H$, there exists $W\in C(\H,\D^\times)$ such that $We_n=\zeta_n$, for every $n \in {\mb N}$.

\end{prop}
\begin{proof} Let $\{\zeta_n\}$ be Bessel-like and $\{e_n\}$ an orthonormal basis for $\H$. For $f\in \H$, $f=\sum_{k=1}^\infty \ip{f}{e_k}e_k$, we define $Wf= \sum_{k=1}^\infty \ip{f}{e_k}\zeta_k$. This series converges in $\D^\times[t^\times]$ as seen in Proposition \ref{prop_210} and it is clear that $We_n=\zeta_n$, for every $n \in {\mb N}$.
We now prove that $W\in C(\H,\D^\times)$. Let us consider a bounded
subset $\cM$ of $\D[t]$; then,
\begin{align*} \sup_{\eta\in \cM}|\ip{Wf}{\eta}| & = \sup_{\eta\in \cM}\left|\ip{ \sum_{k=1}^\infty \ip{f}{e_k}\zeta_k}{\eta} \right|\\
&=\sup_{\eta\in \cM}\left|\sum_{k=1}^\infty \ip{f}{e_k}{\ip{\zeta_k}{\eta} }\right|\\
&\leq \left(\sum_{k=1}^\infty | \ip{f}{e_k}|^2\right)^{1/2} \sup_{\eta\in \cM} \left(\sum_{k=1}^\infty |\ip{\zeta_k}{\eta}|^2\right)^{1/2}\\
&\leq \gamma_\cM ^{1/2}\|f\|.
\end{align*}
Conversely, assume that, given an orthonormal basis $\{e_n\}$, there
exists $W\in C(\H,\D^\times)$ such that $We_n=\zeta_n$. Then, if
$\cM$ is a bounded subset of $\D[t]$,
\begin{align*}
\sup_{\eta\in \cM}\sum_{k=1}^\infty|\ip{\zeta_k}{\eta}|^2 &=
\sup_{\eta\in \cM}\sum_{k=1}^\infty|\ip{We_k}{\eta}|^2 \\ &=
\sup_{\eta\in
\cM}\sum_{k=1}^\infty\left|\ip{e_k}{W^\dag\eta}\right|^2 =
\sup_{\eta\in \cM}\|W^\dag\eta\|^2<\infty.
\end{align*}
Hence $\{\zeta_n\}$ is a Bessel-like sequence.
\end{proof}
}

As in the case of Hilbert spaces, Bessel-like sequences have a dual
counterpart. \bedefi Let $\{\xi_n\}$ be a sequence in $\D$. We say
that $\{\xi_n\}$ is a {\em Riesz-Fischer-like} sequence, if for
every orthonormal basis $\{e_n\}$ of $\H$, there exists $S\in
C(\D,\H)$ such that $S\xi_n=e_n$, for every $n \in {\mb N}$.
\findefi

For an arbitrary sequence $\{\xi_n\}$ in $\D$, we define,  a second
{\em analysis} operator $V$ as follows:

\begin{equation}\label{eq_OPU} \left\{\begin{array}{l} D(V)=\left\{\Phi \in \D^\times: \sum_{k=1}^\infty |\ip{\Phi}{\xi_k}|^2<\infty\right\} \\ V\Phi =\{\ip{\Phi}{\xi_k}\}, \; \Phi \in D(V)  \end{array} \right.\end{equation}

\begin{prop} If $\{\xi_n\}$ is a Riesz-Fischer-like sequence, then $V: D(V)\to \ell^2$ is surjective.
\end{prop}
\begin{proof} Let $\{a_n\}\in \ell^2$ and $\{e_n\}$ an orthonormal basis of $\H$. Put $f= \sum_{k=1}^\infty a_k e_k\in \H$. Then,
$$ a_n= \ip{f}{e_n}=\ip{f}{S\xi_n}=\ip{S^\dag f}{\xi_n}, \quad \forall n \in {\mb N}.$$
Then $\Phi=S^\dag f \in D(V)$ and $V\Phi=\{a_n\}$.
\end{proof}

Let $\{\omega_n\}$ denote the canonical basis in $\ell^2$; i.e.,
$\omega_n=\{\delta_{kn}\}$, for every $n \in {\mb N}$. Then, for
every $n \in {\mb N}$, there exists $\zeta_n \in \D^\times$ (in
general, nonunique) such that $\delta_{kn}=\ip{\zeta_n}{\xi_k},
\quad n,k \in {\mb N}$.

The {\em duality} between Riesz-Fischer-like sequences and
Bessel-like ones is then stated by the following
\begin{prop} \label{prop_215} $\{\xi_n\}$ is a Riesz-Fischer-like sequence in $\D$ if and only if there exists a Bessel-like sequence $\{\zeta_n\}$ in $\D^\times$ such that $\{\xi_n\}$ and $\{\zeta_n\}$ are biorthogonal.
\end{prop}
\begin{proof} Suppose that $\{\xi_n\}$ has a Bessel-like biorthogonal sequence. Then, for every orthonormal basis $\{e_n\}$ in $\H$, there exists $T \in C(\H, \D^\times)$ such that $Te_n=\zeta_n$, for every $n \in {\mb N}$. Then,
$$\delta_{kn}=\ip{\zeta_n}{\xi_k}=\ip{Te_n}{\xi_k}=\ip{e_n}{T^\dag\xi_k} , \quad n,k \in {\mb N}.$$
This easily implies that $T^\dag\xi_k= e_k$, for every $k\in {\mb
N}$.

Conversely, suppose that $\{\xi_n\}$ is a Riesz-Fischer-like
sequence. Then, for every orthonormal basis $\{e_n\}$ in $\H$, there
exists $S \in C(\D, \H)$ such that $S\xi_n=e_n$, for every $n \in
{\mb N}$. Hence,
$$\delta_{kn}=\ip{S\xi_n}{e_k}=\ip{\xi_n}{\S^\dag e_k}.$$
Let us define $\zeta_k=\S^\dag e_k$, $k \in {\mb N}$. Then
$\{\zeta_k\}$ is Bessel-like and $\{\xi_n\}$ and $\{\zeta_n\}$ are
biorthogonal.
\end{proof}

{
In the opposite direction we only get a partial result.
\begin{prop} Let $\{\zeta_n\}$  be a sequence in $\D^\times$. If $\{\zeta_n\}$ possesses a biorthogonal sequence $\{\xi_n\}$ which is total and Riesz-Fischer-like, 
then  $\{\zeta_n\}$ is a Bessel-like sequence.
\end{prop}
\begin{proof} Since $\{\xi_n\}$ is Riesz-Fischer-like,  for every orthonormal basis $\{e_n\}$ in $\H$, there
exists $S \in C(\D, \H)$ such that $S\xi_n=e_n$, for every $n \in
{\mb N}$. Hence,
$$\delta_{kn}=\ip{S\xi_n}{e_k}=\ip{\xi_n}{\S^\dag e_k}.$$
This implies that $$ \ip{\xi_n}{\S^\dag e_k-\zeta_k}=0, \quad \forall k, n \in {\mb N}.$$
Since $\{\xi_k\}$ is total, we conclude that, for every $k \in {\mb N}$, $\zeta_k= S^\dag e_k$. Clearly, $S^\dag \in C(\H, \D^\times)$; hence, the statement follows from Proposition \ref{prop_211}.

\end{proof}
}
Let us now assume that $\{\xi_n\}$ is a Schauder basis for $\D[t]$
and that the dual basis $\{\zeta_n\}$ is a Bessel-like sequence.
Then, $\{\xi_n\}$ is a Riesz-Fischer-like sequence and  by
\eqref{eq_two}, we have, for every $\Phi \in \D^\times$ and for
every bounded subset $\cM$ of $\D[t]$,
\begin{align*}\sup_{f\in \cM}|\ip{\Phi}{f}| &\leq \left( \sum_{k=1}^\infty |\ip{\Phi}{\xi_k}|^2\right)^{1/2}\sup_{f\in \cM} \left( \sum_{k=1}^\infty |\ip{\zeta_k}{f}|^2\right)^{1/2}\\
&\leq \gamma_\cM^{1/2} \left( \sum_{k=1}^\infty
|\ip{\Phi}{\xi_k}|^2\right)^{1/2},
 \end{align*}
which, by putting $\Phi=f$ gives a lower estimate of $\left(
\sum_{k=1}^\infty |\ip{f}{\xi_k}|^2\right)^{1/2}$, similar to that
one gets in the usual formulation in Hilbert spaces.

\section{Riesz-like bases}\label{sect_weak Riesz bases}
\subsection{Basic properties}

\bedefi \label{def_GRBW}A Schauder basis $\{\xi_n\}$ for $\D[t]$ is
called a {\em Riesz-like basis} if there exists a one-to-one
operator $T\in C(\D,\H)$  such that $\{T\xi_n\}$ is an orthonormal
basis for $\H$. \findefi

It is easy to see that every Riesz-like basis is a Riesz-Fisher-like
sequence.

Since $T$ maps $\D[t]$ into $\H[\|\cdot\|]$ continuously,
$T^\dagger$ has a {continuous} extension (which we denote by the
same symbol) from $\H[\|\cdot\|]$ into $\D^\times[t^\times]$. The
range $R(T)$ of $T$  is dense in $\H$ since it contains the
orthonormal basis $\{e_k\}$ with $e_k:= T \xi_k$, $k\in {\mb N}$. In
particular, it may happen that $R(T)=\H$. Hence, the operator
$T^{-1}$ is everywhere defined and it is continuous if, and only if
$R(T^\dag)= \D^\times$. We will name $\{\xi_n\}$ a {\em strict}
Riesz-like basis, in this case. As we shall see in Theorem \ref{thm:
equivalence g R b}, this imposes severe constraints on the topology
$t$ of $\D$.

If $\{\xi_n\}$ is a Riesz-like basis, we can find explicitly the
sequence $\{\zeta_n\} \subset \D^\times$ of Proposition
\ref{prop_2.4.r}. The continuity of $T$ and \eqref{eq three}, in
fact, imply
$$ Tf= \sum_{n=1}^\infty \overline{\ip{\zeta_n}{f}}T\xi_n = \sum_{n=1}^\infty \overline{\ip{\zeta_n}{f}}e_n, \quad \forall f \in \D.$$
This, in turn, implies that
$\overline{\ip{\zeta_n}{f}}=\ip{Tf}{e_n}$, for every $f \in \D.$
Hence $\zeta_n=T^\dagger e_n$, for every $n \in {\mb N}$.

Clearly, for every $n,k\in \mathbb{N}$,
 $$\ip{\zeta_k}{\xi_n} = \ip{T^\dagger e_k}{\xi_n}=\ip{ e_k}{T\xi_n}=\ip{e_k}{e_n}= \delta_{k,n},
$$
and $T^\dagger T\xi_n= \zeta_n$, for every $n \in {\mb N}$. \\

\medskip
{\noindent Moreover, $\{\zeta_n\}$ is a Bessel-like sequence
(Proposition \ref{prop_215}). Indeed, one has, for every bounded
subset $\cM$ of $\D[t]$,
\begin{align*} \sup_{f\in \cM}\sum_{n=1}^\infty |{\ip{\zeta_n}{f}}|^2 &=\sup_{f\in \cM}\sum_{n=1}^\infty \left|{\ip{T^\dagger e_n}{f}}\right|^2\\ &= \sup_{f\in \cM}\sum_{n=1}^\infty |{\ip{e_n}{Tf}}|^2= \sup_{f\in \cM}\|Tf\|^2<\infty.\end{align*}}

\medskip
\noindent An easy computation shows that
\begin{equation} \label{eq_adjoint}T^\dagger g= \sum_{k=1}^\infty d_k \zeta_k \quad \mbox{if} \quad g= \sum_{n=1}^\infty d_n e_n \in \H.\end{equation}

\medskip
\noindent Finally, we have
\begin{equation} \label{eq_range}
T^\dagger(\H)=\left\{\Psi \in \D^\times: \sum_{k=1}^\infty
|\ip{\Psi}{\xi_k} |^2 <\infty\right\},
 \end{equation}
 that is, $T^\dagger(\H)=D(V)$, where $V$ is the operator defined in \eqref{eq_OPU}.

 Indeed, if $\Psi\in T^\dagger(\H)$, then $\Psi= T^\dagger h$, for some $h\in \H$. Let $h=\sum_{k=1}^\infty c_ke_k$. Then, using the continuity of $T^\dagger$,
$$T^\dagger h = T^\dagger\left(\sum_{k=1}^\infty c_ke_k \right) = \sum_{k=1}^\infty c_k T^\dagger e_k  = \sum_{k=1}^\infty c_k\zeta_k .$$
This implies that $c_k= \ip{\Psi}{\xi_k}$ and $\sum_{k=1}^\infty |\ip{\Psi}{\xi_k} |^2 <\infty.$\\
Conversely, let $\sum_{k=1}^\infty \ip{\Psi}{\xi_k}\zeta_k \in
\D^\times$ with $\sum_{k=1}^\infty |\ip{\Psi}{\xi_k} |^2 <\infty.$
Define $h=\sum_{k=1}^\infty \ip{\Psi}{\xi_k} e_k\in \H$. Then
$$ T^\dagger h = \sum_{k=1}^\infty \ip{\Psi}{\xi_k} T^\dagger e_k = \sum_{k=1}^\infty \ip{\Psi}{\xi_k} \zeta_k.$$

\medskip
\noindent The operator $T$ can also be regarded as an Hilbertian
operator (by assumption it maps $\D$ into $\H$).   This operator is
closable in $\H$ if, and only if, the subspace
$$D(T^*)=\{ g\in  \H: T^\dagger g \in \H\}$$ is dense in $\H$. In this case, the operator $T^*$, the adjoint of $T$, is defined by $T^* g= T^\dagger g$, $g\in D(T^*)$.

{\berem If $\{\xi_n\}$ is a Riesz-like basis for $\D[t]$  and
$\{c_n\}\in \ell^2$ with $ \sum_{k=1}^\infty c_k \xi_k=0$, then
$c_n=0$, for every $n \in {\mb N}$.\enrem}

\begin{thm}\label{thm_34} Let $\D[t]$ be complete and reflexive and $\D^\times$ quasi-complete. Let $\{\xi_n\}$ be a {topological basis of} $\D$. The following statements are equivalent.
\begin{itemize}
\item[(i)] $\{\xi_n\}$ is a Riesz-like basis.
\item[(ii)] There exists a unique sequence $\{\zeta_n\}\subset \D^\times$ such that
\begin{itemize}
\item[(ii.a)] $\{\xi_n\}$ and $\{\zeta_n\}$ are biorthogonal;
\item[(ii.b)] for every $f \in \D$,
$\sum_{n=1}^\infty|\ip{\zeta_n}{f}|^2<\infty$;
\item[(ii.c)] the seminorm $p_\zeta$ defined by
$$ p_\zeta(f)= \left( \sum_{k=1}^\infty|\ip{\zeta_n}{f}|^2\right)^{1/2}$$ is continuous on $\D[t]$.
\end{itemize}
\item[(iii)] There exists $S \in \LDD{}$, $S\geq 0$, such that $\{\xi_n\}$ and $\{S\xi_n\}$ are biorthogonal.
\end{itemize}

\end{thm}
\begin{proof} $({\rm i}) \Rightarrow ({\rm ii})$: Let $\{\xi_n\}$ be a Riesz-like basis for $\D$. Then there exists $T\in C(\D, \H)$ such that $\{T\xi_n\}$ is an orthonormal basis for $\H$. Put $e_n=T\xi_n$ and $\zeta_n= T^\dagger e_n$. Then,
$$\ip{\zeta_n}{\xi_k}= \ip{T^\dagger e_n}{\xi_k} = \ip{ e_n}{T\xi_k}=\ip{ e_n}{e_k}=\delta_{n,k},\,\,\, n,k\in\mathbb{N}.$$
It is easily seen that, if $f = \sum_{n=1}^\infty a_n \xi_n$, then
$a_n= \overline{\ip{\zeta_n}{f}}$ and $Tf= \sum_{n=1}^\infty
\overline{\ip{\zeta_n}{f}} e_n$. Hence
$\sum_{n=1}^\infty|\ip{\zeta_n}{f}|^2<\infty$. Moreover, since $T\in
C(\D,\H)$, there exists a continuous seminorm $p$ on $\D[t]$ such
that $\|Tf\|\leq p(f)$, for every $f \in \D$. Hence,
$$p_\zeta(f):=\left( \sum_{k=1}^\infty|\ip{\zeta_n}{f}|^2\right)^{1/2}= \|Tf\|\leq p(f), \quad \forall f \in \D.$$
This implies that $p_\zeta$, which is a seminorm on $\D$, is
continuous.

{$({\rm ii}) \Rightarrow ({\rm iii})$: First, let us define
$S\xi_k= \zeta_k$ and extend $S$ by linearity to \mbox{$\D_0:={\rm
span}\{\xi_m;m\in{\mb N}\}$}. Thus $S:\D_0 \to \D^\times$. If
$f=\sum_{k=1}^n \overline{\ip{\zeta_k}{f}}\xi_k \in \D_0$ and $g=
\sum_{h=1}^\infty \overline{\ip{\zeta_h}{f}}\xi_h \in \D$, we get
\begin{align*} |\ip{Sf}{g}|&= \left|\sum_{k=1}^\infty \overline{\ip{\zeta_k}{f}}\ip{\zeta_k}{g}\right|\\ &\leq  \left(\sum_{n=1}^\infty |\ip{\zeta_n}{f}|^2 \right)^{1/2}\left(\sum_{n=1}^\infty |\ip{\zeta_n}{g}|^2 \right)^{1/2}= p_\zeta(f)p_\zeta(g).\end{align*}
Hence, if $\cM$ is a bounded subset of $\D[t]$, we obtain
$$ \sup_{g\in \cM}|\ip{Sf}{g}|\leq p_\zeta(f)\sup_{g\in \cM}p_\zeta(g).$$
This proves that $S$ is continuous from $\D_0[t]$ into
$\D^\times[t^\times]$.
 Thus { $S$} has an extension (denoted by the same symbol) to a continuous
linear map from the quasi-completion of $\D_0[t]$, which is $\D$, to
the quasi-completion of $\D^\times$, which coincides with
$\D^\times$. Hence, $S \in \LDD$. It is easily seen that
$\ip{Sf}{f}\geq 0$, for every $f\in \D$. }

$({\rm iii}) \Rightarrow ({\rm i})$:
{Since $\{\xi_n\}$ is a topological basis, every $f \in \D$ can be represented, in unique way, as  $f=\sum_{k=1}^\infty a_k\xi_k $}. Let now $S \in \LDD$
be such that $\{\xi_n\}$ and $\{S\xi_n\}$ are biorthogonal. Then the
following equality holds
\begin{equation}\label{eqn_positivity} \ip{Sf}{f}= \sum_{k=1}^\infty |a_k|^2 \; \mbox{ if } f=\sum_{k=1}^\infty a_k\xi_k \in \D. \end{equation}
{This implies that $S\geq 0$ and $\{a_k\}\in \ell^2$. Then, if
$\{e_k\}$ is any orthonormal basis in $\H$ the series
$\sum_{k=1}^\infty a_ke_k$ converges in $\H$. Let us fix one of
these bases $\{e_n\}$.
We define
$$T_n: f= \sum_{k=1}^\infty a_k\xi_k \in \D \to T_nf= \sum_{k=1}^n a_ke_k \in \H$$
and
$$T: f= \sum_{k=1}^\infty a_k\xi_k \in \D \to Tf= \sum_{k=1}^\infty a_ke_k \in \H.$$
{Using \eqref{eqn_positivity} it is easily seen that $T_n\in C(\D,\H)$.  Clearly, $T_nf\to Tf$ in $\H$. }Since $\D[t]$ is reflexive, it is barreled and then,
by the Banach - Steinhaus theorem (see, e.g. \cite[Th.
11.1.3]{jarchow}), it follows that $T\in C(\D,\H)$. Moreover if $f=
\sum_{k=1}^\infty a_k\xi_k \in \D$, then $\|Tf\|^2 =
\sum_{k=1}^\infty |a_k|^2$ whence it follows immediately that $T$ is
injective. By the definition itself, $T\xi_k=e_k$. Therefore
$\{\xi_n\}$ is a Riesz-like basis.}
 \end{proof}

\beex Suppose that $\{e_n\}$ is an orthonormal basis for $\H$ whose
elements belong to $\D$. If $\{e_n\}$ is also a basis for $\D[t]$,
then it is automatically a Schauder basis and since the identity is
continuous from $\D[t]$ into $\H[\|\cdot\|]$, it is clear that
$\{e_n\}$ is a Riesz-like basis for $\D[t]$. The dual sequence in
$\D^\times$ is clearly $\{e_n\}$ itself. This is a familiar
situation. Let us consider, in fact, the triplet ${\mc S}({\mb R})
\subset L^2({\mb R}) \subset {\mc S}^\times({\mb R})$, where ${\mc
S}({\mb R})$ is the Schwartz space of rapidly decreasing
$C^\infty$-functions on the real line and ${\mc S}^\times({\mb R})$
the space of (conjugate) tempered distributions. Then, it is well
known that the set $\{\phi_n\}$ of Hermite functions is not only an
orthonormal basis for $L^2({\mb R})$, but also a basis for $\S({\mb
R})$ in its own topology (see \cite[Theorem V.13]{reedsimon}). \enex

\beex Let $\H$ be a separable Hilbert space and $\{e_n\}$ an
orthonormal basis of $\H$. Let $N$ denote the {\em number} operator
defined on the basis vectors by $N e_k= k e_k$, $k \in {\mb N}$.
Then as it is well-known $N$ is self-adjoint on its natural domain
$$D(N)= \left\{ f \in \H:  \sum_{k=1}^\infty k^2|\ip{f}{e_k}|^2<\infty \right\}.$$
Let $\D:= \D^\infty (N)= \bigcap_{k=1}^\infty D(N^k)$ be endowed
with the topology $t_N$ defined by the seminorms $p_k (\cdot)= \|N^k
\cdot\|$, $k=0, 1,2, \ldots$. Then $\D$ is a Fr\'echet and reflexive
space. Define $\xi_k = \frac{1}{k}e_k, \, k \in {\mb N}$. Clearly,
$N\xi_k = e_k$, for every $k \in {\mb N}$ and $N$ maps continuously
$\D[t_N]$ into $\H$. Moreover, $\{\xi_k\}$ is a basis for $\D[t_N]$.
Indeed, for every $p \in {\mb N}$ we have
\begin{align*} \left\|N^p\left(f- \sum_{k=1}^nk\ip{f}{e_k}\xi_k\right) \right\|&=\left\|N^p  f- \sum_{k=1}^nk\ip{f}{e_k}N^p\xi_k \right\| \\ =\left\|N^p  f- \sum_{k=1}^nk^p\ip{f}{e_k}e_k \right\|\end{align*}
and the latter tends obviously to zero as $n \to \infty$. Since
$\{\xi_k\}$ is  a Schauder basis, it is a Riesz-like basis for
$\D[t_N]$.

\enex \beex Let $\{\xi_n\}$ be a Schauder basis for $\D[t]$. Assume
that there exists a continuous seminorm $p$ on $\D[t]$ such that
$$ \left(\sum_{k=1}^\infty |c_k|^2 \right)^{1/2} \leq p\left(\sum_{k=1}^\infty c_k \xi_k \right),$$
whenever $\sum_{k=1}^\infty c_k \xi_k$ converges in $\D[t]$.

Let $\{e_k\}$ be any orthonormal basis in $\H$. Then the operator
$$ T:f=\sum_{k=1}^\infty c_k \xi_k \to Tf= \sum_{k=1}^\infty c_k e_k$$ is one-to-one and continuous from $\D[t]$ into $\H[\|\cdot\|]$. Clearly $T\xi_k=e_k$, for every $k \in {\mb N}$. Hence, $\{\xi_n\}$ is a Riesz-like basis for $\D[t]$.

\enex

\beex Let $\{\xi_n\}$ be a Schauder basis for $\D[t]$ and $\{\zeta_n
\}$ the corresponding sequence in $\D^\times$ such that
$\ip{\xi_n}{\zeta_m}=\delta_{nm}$. Define $S\xi_n =\zeta_n$, $n \in
{\mb N}$, and assume that $S$ extends to a  positive operator
(denoted by the same symbol) of $\LDD$. If $S= T^\dagger T$, with
$T\in C(\D,\H)$ and $T^\dagger \in C(\H, \D^\times)$, then, as it is
easily seen, the sequence $\{e_n\}$ with $e_n= T\xi_n$ is
orthonormal and, if $T$ is surjective, it is an orthonormal basis
for $\H$. Thus $\{\xi_n\}$ is a Riesz-like basis for $\D[t]$. \enex

\berem It is worth considering the case where, so to say, the rigged
Hilbert space collapses into one Hilbert space only, as it happens
if the topology $t$ of $\D$ is equivalent to the Hilbert norm. Then
$\{\xi_n\}$ is Riesz-like if there exists an invertible bounded
operator $T$ mapping $\{\xi_n\}$ into an orthonormal basis of $\H$.
However, the inverse $T^{-1}$ need not be bounded. Nevertheless the
discussion made so far shows that the essential features of (usual)
Riesz bases in Hilbert space are preserved also in this more general
set-up. \enrem

In the usual definition of Riesz basis in Hilbert space $\H$ one
requires that $\{\xi_n\}$ is mapped into an orthonormal basis of
$\H$ by a bounded operator with bounded inverse. In Definition
\ref{def_GRBW}, we only required the continuity  of the operator
$T$; i.e $T\in C(\D, \H)$. In fact, there is no room for the
continuity of $T^{-1}$ from $\H$ into $\D[t]$, unless $\D[t]$ (and,
then also $\D^\times[t^\times]$) is {\em equivalent} (in topological
sense) to a Hilbert space. We maintain the basic assumption that
$\D[t]$ is complete and reflexive.

\begin{thm}\label{thm: equivalence g R b} Let  $\{\xi_n\}$ be  a sequence of elements of $\D$. The following statements are equivalent.
\begin{itemize}
\item[(i)] $\{\xi_n\}$ is a Riesz-like basis and the one-to-one operator $T\in C(\D,\H)$ for which $\{T\xi_n\}$ is an orthonormal basis of $\H$, has a continuous inverse; i.e. $T^{-1}\in C(\H, \D)$.
\item[(ii)] The space $\D$ can be endowed with an inner product $\ip{\cdot}{\cdot}_{+1}$ such that the topology induced by the corresponding norm $\|\cdot\|_{+1}$ is equivalent to $t$,  $\D[\|\cdot\|_{+1}]$ is a Hilbert space and  the  sequence $\{\xi_n\}$ is an orthonormal basis for $\D[\|\cdot\|_{+1}]$.
\item[(iii)]  The  sequence $\{\xi_n\}$ is complete in $\D[t]$ and there exists a continuous seminorm $p$ such that for every $n \in {\mb N}$ and complex numbers $\{c_1, \ldots, c_n\}$
$$ \sum_{i=1}^n |c_i|^2 \leq p\left( \sum_{i=1}^nc_i \xi_i \right)^2$$ and for every continuous seminorm $q$ there exists $C_q>0$ such that
$$q\left( \sum_{i=1}^nc_i \xi_i \right)^2\leq C_q \sum_{i=1}^n |c_i|^2,$$
for every $n \in {\mb N}$ and complex numbers $\{c_1, \ldots,
c_n\}$.

\end{itemize}
\end{thm}
\begin{proof}$(i)\Rightarrow(ii)$:
 Let $T$ be the continuous  operator with continuous inverse such that $\{T\xi_n\}$ is an orthonormal basis of $\H$ and define
$$\ip{\xi}{\eta}_{+1}:=\ip{T\xi}{T\eta}, \qquad\forall\xi,\eta\in\D.$$ Then, $\ip{\cdot}{\cdot}_{+1}$ is an inner product on $\D$. Let $\|\cdot\|_{+1}$ be the corresponding norm. Clearly, $\|\xi\|_{+1}=\|T\xi\|$, for every $\xi \in \D$. Since $T$ is continuous from $\D[t]$ to $\H$,  there exists a continuous seminorm $p$ such that
\begin{equation}\label{norm +1 and a seminorm}\|\xi\|_{+1}=\|T\xi\|\leq p(\xi),\qquad \forall \xi\in\D.\end{equation}
On the other hand, $T^{-1}$ is continuous from $\H$ onto $\D[t]$,
then  for every seminorm $q$ on $\D$ there exists $\gamma_q>0$ such
that
\begin{equation*}
  q(T^{-1}\zeta)\leq\gamma_q\|\zeta\|, \quad\forall \zeta\in\H.
\end{equation*}
If $\xi\in\D$, then $\xi=T^{-1}\zeta$ for some $\zeta\in\H$, hence
\begin{equation}\label{norm +1 and every seminorm}
  q(\xi)\leq q(T^{-1}\zeta)\leq\gamma_q\|T\xi\|=\gamma_q\|\xi\|_{+1}.\end{equation}

The equivalence of the topology defined by $\|\cdot\|_{+1}$ and $t$
implies that $\D[\|\cdot\|_{+1}]$ is a Hilbert space.

Finally, the sequence $\{\xi_n\}$ is basis consisting of orthonormal
vectors in $\D[\|\cdot\|_{+1}]$; indeed,
\begin{equation}\label{eqn_orthbases}\ip{\xi_i}{\xi_j}_{+1}=\ip{T\xi_i}{T\xi_j}=\ip{e_i}{e_j}=\delta_{ij},\,\,\,\,i,j\in\mathbb{N}.\end{equation}
$(ii)\Rightarrow(iii)$:  Since $\|\cdot\|_{+1}$ defines a topology
equivalent to $t$, then there exists a continuous seminorm $p$ on
$\D[t]$ such that \eqref{norm +1 and a seminorm} holds and for every
continuous seminorm $q$ there exists $\gamma_q>0$ such that
\eqref{norm +1 and every seminorm} holds.

Now, consider any fixed $n \in {\mb N}$ and complex numbers $\{c_1,
\ldots, c_n\}$ and consider the orthonormal basis $\{\xi_n\}$ for
$\D[\|\cdot\|_{+1}]$. If  $\xi=\sum_{i=1}^nc_i \xi_i\in\D$, then
$\|\xi\|_{+1}^2=\sum_{i=1}^n |c_i|^2$
 then the statement follows by applying \eqref{norm +1 and a seminorm} and \eqref{norm +1 and every seminorm} to $\xi$.
Of course the linear span of $\{\xi_n\}$ is dense in $\D[\|\cdot\|_{+1}]$, since $\{\xi_n\}$ is an orthonormal basis for $\D[\|\cdot\|_{+1}]$; hence the sequence $\{\xi_n\}$ is complete in $\D[\|\cdot\|_{+1}]$ and then, by the equivalence of $t$ and of the topology generated by $\|\cdot\|_{+1}$, $\{\xi_n\}$ is complete in $\D[t]$.\\
$(iii)\Rightarrow(i)$: Let $\{e_n\}$ be any orthonormal basis for
$\H$ and define two linear operators $T:\D\to\H$ and $S:\H\to\D$ as
follows: for any fixed $n\in\mathbb{N}$
$T\left(\sum_{i=1}^nc_i\xi_i\right):=\sum_{i=1}^nc_i e_i$ and
$S\left(\sum_{i=1}^n c_i e_i\right):=\sum_{i=1}^nc_i\xi_i$ with
$c_i\in\mathbb{C}$; $T$ and $S$ are continuous;  moreover,
$T\xi_n=e_n$ and $Se_n=\xi_n$, for every $n\in\mathbb{N}$. Certainly
$TS=I$ and, since $\{\xi_n\}$ is complete  in $\D[\|\cdot\|_{+1}]$,
$ST=I\restr{\D}$. Hence, $T$ is a continuous invertible linear
operator  with continuous inverse and $\{\xi_n\}$ is a strict
Riesz-like basis for $\D[t]$.
\end{proof}

The condition given in $(iii)$ is clearly the natural substitute for
the inequalities in \eqref{eqn_bounds} in this setting.

Let us call, for short, {\em strict Riesz-like basis} a basis for
which (i) of Theorem \ref{thm: equivalence g R b} holds.

\begin{prop}\label{prop_4.8} If the rigged Hilbert space $\D[t]\subset \H \subset \D^\times[t^\times]$, with $\D[t]$ complete and reflexive,  has a strict
 Riesz-like basis $\{\xi_n\}$  then it is
(equivalent to) a triplet of Hilbert spaces $\H_{+1}
 \subset \H \subset \H_{-1}$.  Moreover, $\{\xi_n\}$ is an orthonormal basis for $\H_{+1}$ and the dual sequence $\{\zeta_n\}$ is an orthonormal basis for $\H_{-1}$.
\end{prop}
In fact, from the previous discussion, it follows also that
$\H_{+1}=\D$ with norm $\|\xi\|_{+1}=\|T\xi\|$, $\xi \in \D$, where
$T\in C(\D,\H)$ is an operator such that $\{T\xi_n\}$ is an
orthonormal basis for $\H$. This operator $T$, regarded as an
operator in $\H$ is, in general, an unbounded operator with domain
$D(T)=\D$ and bounded inverse.

{Strict Riesz-like bases} have an interest in their own since Riesz
bases in triplets of Hilbert spaces are useful for some applications
\cite{jia}. A more detailed analysis will be given in
\cite{bellom_ct_2}.

\beex Let $A$ be a closed operator in a separable Hilbert space
$\H$, with domain $D(A)$. Then $D(A)$ can be made into a Banach
space, denoted by $\B_A$, if a new norm
 is defined by $\|\varphi\|_A:=\|\varphi\| + \|A\varphi \|.$

Let $\B_A^\times$ be the conjugate dual of $\B_A$ w. r. to
$\|\cdot\|_A$. The operator $(I+A^*A)^{1/2}$ is continuous from
$\B_A$ into $\H$ and its continuous extension to $\H$, denoted by
the same symbol, is continuous from $\H$ into $\B_A^\times$ and has
continuous inverse. If $\{e_n\}$ is an orthonormal basis for $\H$,
then the sequence $\{\xi_n\}$ defined by $\xi_n:=
(I+A^*A)^{-1/2}e_n$ is a strict Riesz-like basis for $\B_A$.

A concrete example can be constructed as follows. Consider the
triplet of Sobolev spaces $W^{1,2}({\mb R}) \subset L^2({\mb R})
\subset W^{-1,2}({\mb R})$. As it is well-known, $W^{1,2}({\mb R}) $
is a Banach space under the norm $\|f\|_{1,2}= \|f\|_2 + \|{\sf
D}f\|_2 $, ${\sf D}$ denoting the weak derivative.

Let $\{\xi_n\}$ be the family of functions of $W^{1,2}({\mb R}) $
defined by
$$ \xi_n(x)= \frac{(-i)^n}{\sqrt{2\pi} }\int_{\mb R} \frac{\phi_n(y)e^{ixy}}{(1+y^2)^{1/2}} dy,$$
where $\phi_n(x)=H_n(x)e^{-x^2/2}$ denotes the $n$-th Hermite
function. The family $\{\xi_n\}$ is a strict Riesz-like basis of
$W^{1,2}({\mb R})[\|\cdot\|_{1,2}] $. In fact it is not difficult to
show by standard techniques of Fourier transform that $ (I-{\sf
D}^2)^{1/2} \xi_n= \phi_n$. Moreover the operator $(I-{\sf
D}^2)^{1/2}$ is continuous from $W^{1,2}({\mb R})$ into $L^2({\mb
R})$ and has continuous inverse. The result of (ii) of Theorem
\ref{thm: equivalence g R b} is not surprising at all. Indeed, as it
is well know, the space $W^{1,2}({\mb R})$ can be made into a
Hilbert space with inner product
$$\ip{\varphi}{\psi}_{1,2}'= \ip{(I-{\sf D}^2)^{1/2}\varphi}{(I-{\sf D}^2)^{1/2}\psi}, \quad \varphi, \psi \in W^{1,2}({\mb R})$$ which endows $W^{1,2}({\mb R})$ with a topology equivalent to that defined by $\|f\|_{1,2}$.

\enex

\subsection{An application}\label{sect_appl}
As mentioned in the Introduction, an important problem of
Pseudo-Hermitian Quantum Mechanics is the following: given a
nonself-adjoint Hamiltonian $ \sf{H}$, with real spectrum, one tries
to find a well-behaved (bounded and with bounded inverse)
intertwining operator $T$ which transforms $\sH$ is a self-adjoint
operator $\sH_{sa}$. When this happens one can get of course a large
amount of information on $\sH$ making use of the spectral theory of
self-adjoint operators. The situation becomes more involved in cases
(like the cubic oscillator) where a so regular operator does not
exist and one has to deal with {\em unbounded} intertwining
operators. Even the notion of similarity must be relaxed, with a
certain loss in the preservation of spectra (see, e.g. \cite{
AnTr2014, JPA_CT1, JPA_CT2}). In this section, we will show how the
approach in rigged Hilbert space can be helpful in these cases.

Let $\sH$ be a closed operator in Hilbert space. As already
mentioned {in Example \ref{ex_triplet}}, its domain $D(\sH)$ can be
made into a Hilbert space $\H_{\sH}$ with the graph norm
$\|\cdot\|_{\sH}$. Let $\H_{\sH}^\times$ be its conjugate dual and
consider the triplet of Hilbert spaces $\H_{\sH}\subset \H \subset
\H_{\sH}^\times$. Assume that $\sH_{sa}$ is a self-adjoint operator
in $\H$ with discrete spectrum and, for simplicity, that every
eigenvalue $\lambda_k \in {\mb R}$ has multiplicity $1$. Let
$\psi_k$ be an eigenvector corresponding to $\lambda_k$. Then
$\{\psi_k\}$ is an orthonormal basis for $\H$. Assume that there
exists $T\in C(\H_{\sH}, \H)$, invertible and with continuous
inverse $T^{-1}:\H \to \H_{\sH}$ such that
\begin{equation}\label{sedici} \ip{\sH \xi}{T^\dag \eta}= \ip{T\xi}{\sH_{sa}\eta}, \quad \forall \xi \in \H_{\sH},\, \eta \in D(\sH_{sa}) \mbox{ s. t. } T^\dag \eta \in \H.\end{equation}
Let us define $\xi_k= T^{-1}\psi_k$, $k\in {\mb N}$. Then, the set
$\{\xi_k\}$ is complete and it is a Schauder basis of
$\H_{\sH}[\|\cdot\|_{\sH}]${(Remark \ref{rem_23}(d))}. Hence it is a
strict Riesz-like basis.
From \eqref{sedici}, for every $\eta \in D(\sH_{sa})$, we get
\begin{align*} \ip{\sH \xi_n}{T^\dag \eta}&= \ip{T\xi_n}{\sH_{sa}\eta}=\ip{\psi_n}{\sH_{sa}\eta}=\ip{\sH_{sa}\psi_n}{\eta}\\ &=\lambda_n\ip{\psi_n}{\eta}=\lambda_n\ip{T\xi_n}{\eta}=\lambda_n\ip{\xi_n}{T^\dag \eta}. \end{align*}
Thus, if { $T^\dag D(\sH_{sa})\cap \H$} is dense in $\H$, we get $\sH\xi_n =\lambda_n \xi_n$, for
every $n\in {\mb N}$.

Conversely, assume that a sequence $\{\xi_n\}$ is a strict
Riesz-like basis for $\H_{\sH}$ and that $\sH\xi_n =\lambda_n
\xi_n$, $\lambda_n \in {\mb R}$, for every $n\in {\mb N}$. Since
there exists an operator $T\in C(\H_{\sH}, \H)$, invertible and with
continuous inverse $T^{-1}:\H \to \H_{\sH}$, such that the vectors
$\psi_n=T\xi_n$ constitute an orthonormal basis for $\H$, we can
construct a self-adjoint operator $\sH_{sa}$, in standard way; i.e.,
 \begin{align*} D(\sH_{sa})&=\left\{\xi \in \H: \sum_{k=1}^\infty \lambda_k^2 |\ip{\xi}{\psi_k}|^2<\infty \right\}\\
\sH_{sa}\xi &= \sum_{k=1}^\infty \lambda_k \ip{\xi}{\psi_k}\psi_k,
\quad \xi \in D(\sH_{sa}).
 \end{align*}
{If $\xi \in \H_{\sH}$, then $\xi= \sum_{k=1}^\infty
\ip{\xi}{\zeta_k}\xi_k$ w. r. to $\|\cdot\|_{\sH}$. This, in particular implies that $\sH\xi= \sum_{k=1}^\infty \lambda_k
\ip{\xi}{{\zeta_k}}\xi_k$, in the norm of $\H$. Then, taking into account that $\xi_k\in\H_{\sH}$, for every $k \in {\mb N}$ and that $T^\dag \in C(\H, \H_{\sH}^\times)$,   we have,
for every $\eta \in D(\sH_{sa}) \mbox{ s. t. } T^\dag \eta \in \H$,
\begin{align*}\ip{\sH \xi}{T^\dag \eta}
&=  \ip{\lim_{N\to \infty}\sum_{k=1}^N \lambda_k \ip{\xi}{\zeta_k}\xi_k}{T^\dag\sum_{r=1}^\infty\ip{\eta}{\psi_r}\psi_r}\\
&= \lim_{N\to \infty} \ip{\sum_{k=1}^N \lambda_k \ip{\xi}{\zeta_k}\xi_k}{T^\dag\sum_{r=1}^\infty\ip{\eta}{\psi_r}\psi_r}\\
&= \lim_{N\to \infty} \ip{\sum_{k=1}^N \lambda_k \ip{\xi}{\zeta_k}T\xi_k}{\sum_{r=1}^\infty\ip{\eta}{\psi_r}\psi_r}\\
&= \ip{\sum_{k=1}^\infty \lambda_k \ip{\xi}{\zeta_k}\psi_k}{\sum_{r=1}^\infty\ip{\eta}{\psi_r}\psi_r}\\
&=\sum_{k=1}^\infty \lambda_k
\ip{\xi}{\zeta_k}\overline{\ip{\eta}{\psi_k}}.
\end{align*}
}

On the other hand,
\begin{align*}
\ip{T\xi}{\sH_{sa}\eta}&= \ip{\sum_{k=1}^\infty\ip{\xi}{\zeta_k}\psi_k}{\sum_{r=1}^\infty\lambda_r\ip{\eta}{\psi_r}\psi_r}\\
&=\sum_{k=1}^\infty \lambda_k
\ip{\xi}{\zeta_k}\overline{\ip{\eta}{\psi_k}}.
\end{align*}
Hence the {\em weak} similarity condition \eqref{sedici} is
fulfilled. It is clear that in what we have done a crucial role is
played by the continuity of both $\sH$ and $T$ as linear maps from
$\H_{\sH}$ into $\H$, even though they are in general unbounded
operators when regarded in $\H$. It is worth pointing out that the
assumption $T\in C(\H_{\sH}, \H)$ does not imply that $T$ is a
closable operator in $\H$. But, requiring that $\{ \eta \in
D(\sH_{sa}) \mbox{ s. t. } T^\dag \eta \in \H\}$ is dense in $\H$,
implies that $T$ has a densely defined {\em hilbertian} adjoint
$T^*$ and so it is automatically closable.

\subsection*{Acknowledgement} This work has been supported by the
Gruppo Nazionale per l'Analisi Matematica, la Probabilit\`{a} e le
loro Applicazioni (GNAMPA) of the Istituto Nazionale di Alta
Matematica (INdAM).

\end{document}